\documentclass[12pt,a4paper,english,reqno]{amsart}
\usepackage[a4paper, margin= 1in]{geometry}
\usepackage{amsmath}
\usepackage{amssymb}
\usepackage{amsthm}
\usepackage{graphicx}
\usepackage{gensymb}
\usepackage{matlab-prettifier}
\usepackage[english]{babel}
\usepackage{tikz}
\usepackage{cite}
\usepackage{hyperref}
\usepackage{float}
\usepackage[utf8]{inputenc} 
\usepackage{dsfont}

\newcommand{\norm}[1]{\left \lVert #1 \right \rVert}

\newenvironment{proofsketch}[1][Sketch of Proof]{%
  \begin{proof}[#1]%
}{%
  \end{proof}%
}

\newtheorem{theorem}{Theorem}[section]
\newtheorem*{theorem*}{Theorem}
\newtheorem{lemma}[theorem]{Lemma}
\newtheorem{corollary}[theorem]{Corollary}
\newtheorem{question}[theorem]{Question}

\newtheorem{definition}[theorem]{Definition}
\theoremstyle{remark}
\newtheorem*{remark}{Remark}
\title{A Generalisation of the Concentration-of-Measure Phenomenon with Applications to Intersection Problems}
\author{Benjamin Gillott*}
\begin{document}
\maketitle
\begingroup
\renewcommand{\thefootnote}{\fnsymbol{footnote}}
\footnotetext[1]{Trinity College, University of Cambridge, United Kingdom. Email: \href{mailto:bg449@cam.ac.uk}{bg449@cam.ac.uk}.}
\endgroup
\setcounter{section}{-1}

\begin{abstract}
In this paper we prove a generalisation of the concentration-of-measure phenomenon in the discrete cube. In this setting, the concentration-of-measure phenomenon states that for every subset $\mathcal{A}$ of the discrete cube, its sum with a Hamming ball of suitably large radius $r$ — or equivalently, its $r$-expansion — results in a substantial increase in measure. We define a notion of `$(\gamma,C)$-well-spread' for subsets of the discrete cube $\{0,1\}^n$ for which the following holds: for all $\epsilon$, there exist constants $\gamma$ and $C$ such that for every $\mathcal{A}$ with $|\mathcal{A}| \geq \epsilon2^n$ and every $(\gamma,C)$-well-spread $S$, $|\mathcal{A} + S|$ is at least $(1-\epsilon)2^n$.
\\
\\
We use this result to prove new non-trivial upper bounds to two intersection problems: how many subsets (or subgraphs) can one take from $[n]$ or $[\binom{n}{2}]$ such that every pair's intersection contains some given substructure? We prove non-trivial upper bounds for the $C_4$-intersection problem and the $4$-AP-intersection problem. We also give upper bounds that tend to $0$ for the $H$-intersection problem and $k$-AP-intersection problem as the number of edges and $k$ tend to infinity. Previously, non-trivial upper bounds were only known for non-bipartite $H$ and nothing was known for the $k$-AP-intersection problem.
\end{abstract}

\section{Introduction}
Within extremal combinatorics, intersection problems are a classical area of study, dating back to 1938 when Erdős, Ko and Rado proved the now famous Erdős-Ko-Rado Theorem \cite{EKR}:
\begin{theorem}[Erdős, Ko and Rado]
Let $k$ and $n$ be such that $k \leq n/2$. If $\mathcal{A} \subset [n]^{(k)}$ is a family of sets such that every pair of sets in $\mathcal{A}$ have a non-empty intersection, then:
$$
|\mathcal{A}| \leq \binom{n-1}{k-1}
$$
\end{theorem}
Since then, many interesting problems have been posed and answered with various techniques. See Ellis \cite{EllisOverview} for an overview of the subject.
\\
\\
For a family of sets $\mathcal{A} \subset \{0,1\}^n$ let $\mu(\mathcal{A}) = |\mathcal{A}|/2^n$. The intersection problems we are concerned with can be posed in a general framework as follows: given a class of sets $\mathcal{C} \subset \{0,1\}^n$, what is the largest family of sets $\mathcal{A}$ such that every pair of sets $A$ and $B$ in $\mathcal{A}$ intersect in an element of $\mathcal{C}$? In the case that $\mathcal{C}$ does not contain the empty set we have the upper bound of $\mu(\mathcal{A}) \leq 1/2$ which follows from the simple fact that we cannot take both a set and its complement. We will refer to this upper bound as the trivial upper bound. In this paper, $\mathcal{A}$ is always used to denote the family of sets that intersect in elements of $\mathcal{C}$, for whichever $\mathcal{C}$ we are currently concerned with.
\\
\\
Simonovits and Sós posed the following pair of problems in 1976, which are about what happens when the ground set has additional structure. They asked about the maximal size of a family of subsets of $[n]$ such that every pair of sets in the family intersects in a set containing a $3$-AP. They also asked for the maximal size of a family of labelled graphs on $n$ vertices such that every pair of graphs in the family intersects in a graph containing a triangle (these families are referred to as triangle-intersecting). For both of these questions Simonovits and Sós conjectured that the optimal construction was to take a specific $3$-AP or triangle and take all sets that contain this triple, which would imply that $\mu({\mathcal{A}}) \leq 1/8$. This lower-bound construction, where we specify a fixed set and take all supersets, is referred to as the Erdős-Ko-Rado construction, and is in many cases conjectured to be optimal. The problem for $3$-APs remains completely open: neither of the trivial bounds has been improved on. However, in a significant breakthrough in 2012, Ellis, Friedgut and Filmus \cite{T-intersecting} resolved the triangle-intersecting conjecture in the affirmative.
\\
\\
Generalisations of the triangle-intersecting problem have also been considered. For a fixed graph $H$, we define $m(H) = \sup\{\mu{(\mathcal{A}}) \mid \mathcal{A}\text{ is } H \text{-intersecting}\}$. Alon and Spencer \cite{Alon} asked for which graphs is it the case that $m(H) = 1/2$. Alon remarked that this holds for $H$ a disjoint collection of stars. (To see this, one takes a disjoint collection of $k$ stars each with $m$ edges and takes the set of all graphs with at least $m/2 + C$ edges on at least $k/2 + C$ of these stars. For every fixed $C$, with $k = o(m)$ and $k$ tending to infinity, we have that $\mu(\mathcal{A}) = 1/2 - o(1)$.) Alon conjectured that $m(P_3) < 1/2$. If this were true it would completely determine which graphs $H$ have $m(H) = 1/2$. There has also been interesting work by Berger and Zhao \cite{Zhao} extending the result of Ellis, Friedgut and Filmus to $K_4$-intersecting families.
\\
\\
Chung, Graham, Frankl and Shearer \cite{Shearer} showed that $m(H) \leq 1/4$ for every $H$ that is not bipartite. Our contribution is to prove a non-trivial upper bound for $m(C_4)$. Previously, no upper bound better than the trivial upper bound was known for any bipartite graph. We actually prove a stronger statement. In Theorem 3.2 we show that there exists a positive $\epsilon$ such that every cycle-intersecting family of graphs has measure at most $\frac{1}{2}-\epsilon$. We also show in Theorem 2.2 that $m(K_{3,t})$ is at most $t^{-\frac{1}{6}}$ for $t$ sufficiently large. We mention here some improvements to lower bounds for $m(H)$. Christofides \cite{Christofides} showed that $m(P_3) \geq 17/128$ — the first non-trivial lower bound proved for any $H$. More recently Balogh and Linz improved on the Erdős-Ko-Rado construction for a wider class of bipartite graphs \cite{Balogh}.
\\
\\
We also prove upper bounds for the $k$-AP-intersection problem. We show in Theorem 3.3 that there exists $\epsilon > 0$ such that every $4$-AP-intersecting family has measure at most $\frac{1}{2}-\epsilon$. Previously, nothing non-trivial was known for any $k$. In Theorem 2.1 we show that if $\mathcal{A}$ is $k$-AP-intersecting then $\mu(\mathcal{A}) \leq k^{-\frac{1}{3}}$ for $k$ sufficiently large.
\\
\\
Our proofs rest on a generalisation of the concentration-of-measure phenomenon in the discrete cube — in what follows we will motivate our generalisation in the context of intersection problems. For two sets $A$ and $B$, $A+B$ will denote the sum viewing them as vectors over $\mathbb{F}_2$. For set systems $\mathcal{A}$ and $\mathcal{B}$, we write $\mathcal{A} + \mathcal{B}$ for the family $\{A+B \mid A \in \mathcal{A}, B \in \mathcal{B}\}$. The concentration-of-measure phenomenon states that for all $\epsilon>0$, there exists a $C$ such that for every $\mathcal{A}$ with measure at least $\epsilon$, $\mathcal{A}+B(C\sqrt{n})$ (where $B(r)$ is the Hamming ball of radius $r$) has measure at least $1-\epsilon$. The relation between intersection problems and concentration-of-measure phenomena is classical. In 1964 Katona \cite{Katona} answered a question of Erdős, Ko and Rado about the maximal family of subsets of $[n]$ such that every pair intersects in at least $t$ elements. We present his argument (for $n$ and $t$ of the same parity) here as it serves as motivation for our proofs in this paper. 

\begin{theorem}[Katona]
Let $n+t$ be even. Let $\mathcal{A} \subset \mathcal{P}([n])$ be a set-system such that for every $A$ and $B$ in $\mathcal{A}$ we have $|A \cap B| \geq t$. Then, $$|\mathcal{A}| \leq \sum_{i=0}^{\frac{n-t}{2}} \binom{n}{i}.$$
\end{theorem}
\begin{proofsketch}
We have that for every $x \in \mathcal{A}$ and $y \in B(t-1)$, $x^c + y \notin \mathcal{A}$. Since otherwise $x \cap (x^c + y) = y \cap x \subset y$ and therefore has size less than $t$. Let $\mathcal{A}^c$ denote $\{x^c \mid x \in \mathcal{A}\}$. It holds that $(\mathcal{A}^c + B(t-1)) \cap \mathcal{A} = \emptyset$. This implies that $|\mathcal{A}^c + B(t-1)| \leq 2^n - |\mathcal{A}|$. Harper's inequality \cite{Harper} then gives us the desired bound.
\end{proofsketch}

Note that for $t \sim \sqrt{n}$, we obtain a constant improvement to the upper bound for $\mu(\mathcal{A})$, an illustration of the concentration-of-measure phenomenon. For the sake of clarity, we will focus on the $k$-AP-intersection problem. However, all that follows also applies to the bipartite-graph question. If we try to follow Katona's argument we would need the addition of the $k$-AP-free sets to a family $\mathcal{A}$, with $\mu(\mathcal{A}) \approx 1/2$, to increase the measure by an amount independent of $\mathcal{A}$ and $n$. We notice that a $1-o(1)$ proportion of the sets of size $\sqrt{n}$ do not have a $k$-AP for any $k \geq 5$. This motivates asking how we can extend the concentration-of-measure phenomenon when instead of adding the full Hamming ball we add only a subset with certain properties. 
\\
\\
It is straightforward to observe that simply asking for a $1-o(1)$ proportion of the sets of size $\sim\sqrt{n}$ is not sufficient: if one takes both $\mathcal{A}$ and $S$ to be the collection of sets without some fixed element (this is a rather degenerate example), then $\mu({\mathcal{A}}) = 1/2$ and $S$ contains all bar a $\sim 1/\sqrt{n}$ proportion of the sets of size $\sim\sqrt{n}$. However, $\mathcal{A} + S = \mathcal{A}$ so we do not have a concentration-of-measure phenomenon for this $S$.
\\
\\
In the following section we define a notion of a `$(\gamma,C)$-well-spread' $S$ that roughly states that regardless of how we weight the importance of the coordinates, $S$ contains most of the analogue of $B(\sqrt{n})$ with respect to that weighting. For example, it will imply that for every subset $J \subset [n]$ of size $k$, we have most subsets of $J$ with size at most $\sqrt{k}$. We then show in Corollary 1.3 that we obtain the concentration-of-measure phenomenon for every $(\gamma,C)$-well-spread $S$ (where $\gamma$ and $C$ will depend on $\epsilon$). We deduce this result as an easy consequence of our main result, Theorem 1.1, which in some sense is a probabilistic generalisation of the concentration-of-measure phenomenon.
\\
\\
It is then quite straightforward to obtain non-trivial bounds for the intersection problems for sufficiently large bipartite graphs and long APs — all one must do is check that $H$-free sets and $k$-AP-free sets are sufficiently well spread. We do this in Section 2.
\\
\\
In Section 3 we prove a specialised version of Theorem 1.1 tailored to the cycle and $4$-AP intersecting problem and use it to deduce non-trivial upper bounds for those problems. In some sense these are stronger results than those in Section 2, at least from the perspective of Alon's question concerning which graphs have a non-trivial upper bound. We have chosen to present our `weaker' results on the graph-intersection problem first since we wish to highlight Theorem 1.1. The concentration-of-measure result we use in Section 3 is narrower in scope and more technical to prove, so it might obscure some of the ideas if it were presented first.

\section{The Main Result}
Before stating our main result, we must first make some simple definitions. Throughout this paper, we typically identify a set with its indicator function from $[n]$ to $\{0,1\}$. For $P \in [0, 1]^n$ we take $Y = Y_{P}$ to be the random vector in $\{0,1\}^n$ such that each $Y_i$ is a Bernoulli random variable with $\mathbb{P}(Y_i = 1) = P_i$ and the $Y_i$ are independent from each other. We also let $B(x)$ denote the set of all $y \in \{0,1\}^n$ such that $y_i \leq x_i$ for all $i$. 
\begin{theorem}
For all $\epsilon, \delta > 0$ and every positive integer $n$, given any $\mathcal{A} \subset \{0,1\}^n$, for at least a $1-\epsilon$ proportion of $x \in \{0,1\}^n$ the following holds:
\\
There exists a $P \in [0,1]^n$ such that $\norm{P}_2 ^2 \leq \frac{\log(\frac{2}{\epsilon})}{2\delta\epsilon}$ satisfying $\mathbb{P}(x+Y_P \in \mathcal{A}) \geq \mu(\mathcal{A}) -\delta$.

\end{theorem}
Theorem 1.1 is roughly stating that for almost all $x$, we can find a small ball around $x$ such that the density of $\mathcal{A}$ in this ball is roughly the same \footnote{Theorem 1.1 only gives a lower bound on the density of $\mathcal{A}$ in this ball. By considering Theorem 1.1 applied for $\mathcal{A}$ and $\mathcal{A}^c$, we see that for at least a $1-2\epsilon$ proportion of $x$ we have a $P$ for $\mathcal{A}$ and a $Q$ for $\mathcal{A}^C$ that satisfy our conclusion. By considering the linear interpolation of $P$ and $Q$ we see there is a $R$ such that $\mu(\mathcal{A}) -\delta \leq \mathbb{P}(X+Y_R \in \mathcal{A}) \leq \mu(\mathcal{A}) + \delta$ with $\norm{R}_2^2 \leq \frac{\log(\frac{2}{\epsilon})}{2\delta\epsilon}$.} as the density of $\mathcal{A}$ in the cube. Here, `small ball around $x$' is taken to mean that there is noise $Y_P$ which we add to $x$ such that $\left \| P \right \|_2 ^2$ is small. Thus, we have a probability distribution on $\{0,1\}^n$ which is concentrated around $x$ and `sees' the true density of $\mathcal{A}$. The $\ell_2$ norm is the natural measure of size in this context, motivated by the concentration-of-measure phenomenon. Consider for example the families $\mathcal{A}_J$, for each $J$ a non empty subset of $[n]$, defined to be the set of vectors in $\{0,1\}^n$ with $\sum_{i \in J} x_i \geq \frac{k}{2} + \sqrt{k}$. Note that $\mathcal{A}_J$ has measure uniformly bounded away from 0 and for $x$ with $\sum_{i \in J} x_i \leq  \frac{k}{2}$ we must alter $\sim\sqrt{k}$ coordinates in $J$ in order to be in $\mathcal{A}_J$. So one would take $P_i \sim \frac{1}{\sqrt{|J|}}$ for $i$ in $J$ and $0$ otherwise.
\\
\\
There are also examples to show that $P$ must in general depend on $x$. Consider the tribes example of Ben-Or and Linial \cite{Tribes}: we let $n$ equal $rs$ and consider $[n]$ as $X_1 \cup X_2 ... \cup X_s$ where each $X_i$ has size $r$ and they are mutually disjoint. Then we take $x$ to be an element of $\mathcal{A}$ if it contains some $X_i$. For $r$ chosen suitably (about $\log_2(n)$) we get $\mu(\mathcal{A}) \sim 1/2$. It is easy to see that a $1-o(1)$ proportion of elements not in $\mathcal{A}$ have some $X_i$ for which all bar one coordinate, $j$, is $1$. So in Theorem 1.1 one simply takes $P_j = 1$ and all other entries to be $0$. However, a straightforward calculation shows that if one chooses the same $P$ (with bounded $\ell_2$ norm) for all $x$ then we have that only a $o(1)$ proportion of elements not in $\mathcal{A}$ satisfy the conclusion of Theorem 1.1.
\\
\\
Before we prove Theorem 1.1 we will deduce a simple corollary that extends the concentration-of-measure phenomenon to `well-spread' subsets of the discrete cube.

\begin{definition}[$(\gamma,C)$-well-spread]
A subset $S$ of $\{0,1\}^n$ is $(\gamma,C)$-well-spread if for every $P \in [0,1]^n$ with $\left \| P \right \|_2 ^2 \leq C$, we have $\mathbb{P}(Y_P \in S) > 1-\gamma$.
\end{definition}

For example, the Hamming ball of radius $C\sqrt{n}$ in $\{0,1\}^n$ is $(\gamma,C')$-well-spread for any $C' < C^2$, $\gamma > 0$ and $n$ sufficiently large. To see this, simply observe that by Cauchy-Schwarz for every $P \in [0,1]^n$ with $\left \| P \right \|_2 ^2 \leq C'$, $\sum_{i=1}^{n} P_i \leq C'^{1/2}\sqrt{n}$. So we have that $\mathbb{E}(\sum_{i=1}^{n} (Y_P)_i) \leq C'^{1/2}\sqrt{n}$ and $\mathrm{Var}(\sum_{i=1}^{n} (Y_P)_i) \leq C'^{1/2}\sqrt{n}$. Hence, by Chebyshev's inequality, $\mathbb{P}(Y_P \in B(C\sqrt{n}))$ tends to $1$ as $n$ tends to infinity. 
\\
\\
For an example of a set that is not, for any $\gamma<1$ and $C$, $(\gamma,C)$-well-spread for sufficiently large $n$, consider the family of subsets of $[n]$ which do not contain a $3$-AP. To see this, take $P = (\sqrt{C/n},...,\sqrt{C/n})$. It is easy to see that with high probability $Y_P$ contains a $3$-AP. This is why our methods are unable to prove non-trivial upper bounds for the $3$-AP-intersection problem.

\begin{corollary}
For $\mathcal{A}$ a subset of $\{0,1\}^n$ with $\mu(\mathcal{A}) \geq \epsilon > 0$. Let $C = \frac{\log(\frac{2}{\epsilon})}{\epsilon^2}$. We have that for every $(\epsilon/2,C)$-well-spread subset $S$, $\mu(\mathcal{A} + S) \geq 1 - \epsilon$.
\end{corollary}

\begin{proof}
Apply Theorem 1.1 with $\mathcal{A}$ and $\delta = \epsilon/2$. Then for all $x$ given by Theorem 1.1 with a $P$ such that $\left \| P \right \|_2 ^2 \leq \frac{\log(\frac{2}{\epsilon})}{\epsilon^2}$ and $\mathbb{P}(x+Y_P \in \mathcal{A}) \geq \mu(\mathcal{A}) -\delta \geq \epsilon/2$. As $S$ is $(\epsilon/2,C)$-well-spread $\mathbb{P}(Y_P\in S) > 1- \epsilon/2$ so there is an instance where $Y_P$ is in $S$ and $x+Y_P$ is in $\mathcal{A}$ so $x \in \mathcal{A} + S$.
\end{proof}

\begin{remark}
Since $B(C\sqrt{n})$ is $(\gamma,C')$-well-spread for any $C' < C^2$, $\gamma > 0$ and $n$ sufficiently large we have recovered the concentration-of-measure phenomenon.
\end{remark}

The idea behind the proof of Theorem 1.1 is that we reveal the coordinates of $x$ one by one (sampling $x$ uniformly at random in the discrete cube). We define $P_i$ by a process such that $P_i$ depends only on the first $i$ coordinates. We set all $P_i$ to initially equal $0$. We then after revealing the first $i$ coordinates view the expected probability that $x+Y_P$ is in $\mathcal{A}$ conditional on what we have seen. If in the next step the value of $X_{i+1}$ affects this a lot, we take $Z_{i+1}$ to be value that increases the expected value, we choose $P_{i+1}$ to be large if $X_{i+1} \neq Z_{i+1}$ otherwise we take $P_{i+1} = 0$ since we are happy with the value of $X_{i+1}$. If it has a small effect, we choose $P_{i+1}$ to be small as we do not care about the value of $X_{i+1}$. Then the expected probability — averaging over the unrevealed $X$ — that $X+Y_P$ is in $\mathcal{A}$ is a submartingale with respect to the natural filtration. We show that the bias upwards we have given it is enough so that with high probability we do not end too low, and the bias upwards is small enough such that with high probability we have chosen $\left \| P \right \|_2 ^2$ small.

\begin{proof}[Proof of Theorem 1.1]
Given an $\mathcal{A} \subset \{0,1\}^n$. We will consider $X = (X_i)_i$, a random variable distributed uniformly on $\{0,1\}^n$. To avoid confusion, we will use $\mathbb{P}_X$ to denote the probability with respect to this random variable and $\mathbb{P}_P$ to denote the probability considered in the statement of the theorem (and similarly with expectation). We will define $P$ such that $P_i$ depends only on $(X_1,..,X_i)$. We will show that the expected value of $\sum_{i=1}^{n} P_i^2$ is at most $\frac{\log(\frac{2}{\epsilon})}{4\delta}$. We will also show that $\mathbb{P}_X(\mathbb{P}_P(X+Y_P \in \mathcal{A}) \geq \mu(\mathcal{A}) -\delta)$ is at least $1-\frac{\epsilon}{2}$. So by applying Markov's inequality to $\left \| P \right \|_2 ^2$ and taking a union bound on the events $\{\left \| P \right \|_2 ^2 > \frac{\log(\frac{2}{\epsilon})}{\delta\epsilon}\}$ and $\{\mathbb{P}_P(X+Y_P \in \mathcal{A}) < \mu(\mathcal{A}) -\delta\}$ we may conclude.
\\
\\
Having revealed $(X_1,...,X_k)$ to equal $(x_1,...,x_k)$ and defined $(P_1,...,P_k)$, let $Q$ denote the element of $[0,1]^n$ equal to $P_i$ for $i \leq k$ and $0$ otherwise. We define $W_k$ to equal:  
$$\mathbb{E}_X(\mathbb{P}_Q(X+Y_Q \in \mathcal{A}) \mid (X_i)_{i \leq k} = (x_i)_{i \leq k})$$
We then have that $W_k = \frac{1}{2}\mathbb{E}_X(\mathbb{P}_Q(X+Y_Q \in \mathcal{A}) \mid (X_i)_{i \leq k} = (x_i)_{i \leq k}, X_{k+1} = 0)+\frac{1}{2}\mathbb{E}_X(\mathbb{P}_Q(X+Y_Q \in \mathcal{A}) \mid (X_i)_{i \leq k} = (x_i)_{i \leq k}, X_{k+1} = 1)$. 
\\
\\
We define $O_{k+1}$ to equal: 
$$\vert\mathbb{E}_X(\mathbb{P}_Q(X+Y_Q \in \mathcal{A}) \mid (X_i)_{i \leq k} = (x_i)_{i \leq k}, X_{k+1} = 0)-W_k|$$
Note that $O_{k+1}$ depends only on $(X_1,...,X_k)$ so is independent of $X_{k+1}$. We define $Z_{k+1}$ to denote the choice of $0$ or $1$ that maximises $\mathbb{E}_X(\mathbb{P}_Q(X+Y_Q \in \mathcal{A}) \mid (X_i)_{i \leq k} = (x_i)_{i \leq k}, X_{k+1} = Z_{k+1})$. Lastly, we define $P_{k+1}$ to equal $0$ if $X_{k+1} = Z_{k+1}$ and to equal $\min(CO_{k+1},1)$ if $X_{k+1} \neq Z_{k+1}$. We have set $C = \frac{\log(\frac{2}{\epsilon})}{2\delta}$ to ease notation.
\\
\\
By definition, $W_0 = \mu(\mathcal{A})$ and $W_n$ is $\mathbb{P}_P(X+Y_P \in \mathcal{A})$. Let $D_{k+1} = W_{k+1}-W_{k}$. Note that $W_0,W_n \in [0,1]$, and hence that $\sum_{i=1}^{n} D_i \leq 1$.
\\
\\
The following holds: $$D_{k+1} = O_{k+1}1_{\{X_{k+1} = Z_{k+1}\}} - (O_{k+1} - 2P_{k+1}O_{k+1})1_{\{X_{k+1} \neq Z_{k+1}\}}$$ This follows since $Q = (P_1,...,P_{k+1},0,...,0)$, if $X_{k+1} = Z_{k+1}$. Then $(Y_{Q})_{k+1}$ equals $0$, so we have that $\mathbb{E}_X(\mathbb{P}_Q(X+Y_Q \in \mathcal{A}) \mid (X_i)_{i \leq k+1} = (x_i)_{i \leq k+1})$ is just $W_k + O_{k+1}$ by the definition of $O_{k+1}$. In the case where $X_{k+1} \neq Z_{k+1}$, $\mathbb{E}_X(\mathbb{P}_Q(X+Y_Q \in \mathcal{A}) \mid (X_i)_{i \leq k+1} = (x_i)_{i \leq k+1})$ equals $P_{k+1}(W_i+O_i) + (1-P_{k+1})(W_i-O_i)$ as all the $(Y_Q)_i$ are mutually independent. 
\\
\\
Hence $\mathbb{E}_X(D_{k+1}) = \mathbb{E}_X(O_{k+1}1_{\{X_{k+1} = Z_{k+1}\}} - (O_{k+1} - 2P_{k+1}O_{k+1})1_{\{X_{k+1} \neq Z_{k+1}\}}) = \mathbb{E}_X(2P_{k+1}O_{k+1})$, using the fact that $O_{k+1}$ is independent from $X_{k+1}$ and that $P_{k+1} = 0$ if $X_{k+1} = Z_{k+1}$. So we have that, since $O_{k+1} \geq C^{-1}P_{k+1}$, $2C^{-1}\mathbb{E}_X(\sum_{i=1}^{n} P_i^2) \leq \mathbb{E}_X(\sum_{i=1}^{n} D_i) \leq 1$. So the expected value of $\sum_{i=1}^{n} P_i^2$ is at most $\frac{\log(\frac{2}{\epsilon})}{4\delta}$.
\\
\\
To show that $\mathbb{P}(W_n < \mu(A) - \delta) \leq \epsilon/2$, we first prove a very short and simple lemma in order to bound the expected value of $e^{-2C(W_n-\mu(\mathcal{A}))}$.

\begin{lemma}
For all $x \geq 0$, $\frac{1}{2}(e^{-x} + e^{x-x^2}) \leq 1$.
\end{lemma}

\begin{proof}
Let $f(x) = e^{-x} + e^{x-x^2}$. As $f(0) = 2$, it suffices to prove that $f'(x) \leq 0$ for all $x$. $f'(x) = -e^{-x}+(1-2x)e^{x-x^2} \leq -e^{-x}(1-e^{-x^2}) \leq 0$. Here we have used that, for all $x$, $1-2x \leq e^{-2x}$.
\end{proof}

We have that $\mathbb{E}_X(e^{-2C(W_n-\mu(\mathcal{A}))}) = \mathbb{E}_X(e^{-2C\sum_{i=1}^{n} D_i})$. We show inductively, that for all $k$, $\mathbb{E}_X(e^{-2C\sum_{i=1}^{k} D_i})$ is at most $1$. We can take the base case to be $k=0$, which is trivial. For the inductive step, let $P_{k+1}'$ denote $\min(CO_{k+1},1)$, i.e the possible value that $P_{k+1}$ may take. We have that:
\begin{align*}
\mathbb{E}_X(e^{-2C\sum_{i=1}^{k+1} D_i}) &= \frac{1}{2}\mathbb{E}_X(e^{-2C\sum_{i=1}^{k+1} D_i} \mid X_{k+1} = Z_{k+1}) + \frac{1}{2}\mathbb{E}_X(e^{-2C\sum_{i=1}^{k+1} D_i} \mid X_{k+1}\neq Z_{k+1})\\
&=\frac{1}{2}\mathbb{E}_X(e^{-2C\sum_{i=1}^{k} D_i}e^{-CO_{k+1}}) + \frac{1}{2}\mathbb{E}_X(e^{-2C\sum_{i=1}^{k} D_i}e^{2CO_{k+1}-4CO_{k+1}P_{k+1}'})\\
&=\mathbb{E}_X(e^{-2C\sum_{i=1}^{k} D_i}\frac{1}{2}(e^{-2CO_{k+1}}+e^{2CO_{k+1}-4CO_{k+1}P_{k+1}'}))
\end{align*}
We have that $\frac{1}{2}(e^{-2CO_{k+1}}+e^{2CO_{k+1}-4CO_{k+1}P_{k+1}'})$ is always at most $1$, as if $P_{k+1}' = 1$ then both terms are less than $1$ and otherwise we can apply our lemma with $x = 2CO_{k+1}$. So our induction is complete.
\\
\\
Applying Markov's inequality to $e^{-2C(W_n-\mu(\mathcal{A}))}$ we obtain that $\mathbb{P}_X(W_n < \mu(\mathcal{A}) -\delta)<e^{-2C\delta} = \frac{\epsilon}{2}$. This completes the proof.
\end{proof}

We can improve $\mathbb{P}((x+B(Y_P)) \cap \mathcal{A} \neq \emptyset)$ at the expense of $\norm{P}_2$ by a standard argument. Note that $\mathbb{P}((x+B(Y_P)) \cap \mathcal{A} \neq \emptyset)$ is at least $\mathbb{P}(x+Y_P \in \mathcal{A})$. If we define $Q$ pointwise to equal $2P-P^2$, we then have that $B(Y_Q)$ is distributed identically to the sum of two independent copies of $B(Y_P)$. Hence, $\mathbb{P}((x+B(Y_Q)) \cap \mathcal{A} \neq \emptyset) \geq 2\mathbb{P}((x+B(Y_P)) \cap \mathcal{A} \neq \emptyset)-\mathbb{P}((x+B(Y_P)) \cap \mathcal{A} \neq \emptyset)^2$. As $Q \leq 2P$ pointwise, $\norm{Q}_2 \leq 2\norm{P}_2$. Bootstrapping this, we can see that at the expense of a factor of roughly  $\mu(\mathcal{A})^{-1}$ in $\norm{P}_2$ we can instead ask for $\mathbb{P}((x+B(Y_P)) \cap \mathcal{A} \neq \emptyset)$ to be at least any fixed constant less than $1$.
\\
\\
Before moving on to the applications concerning the intersection problems we will show a straightforward application of Theorem 1.1 to prove Talagrand's inequality \cite{Talagrand} in the discrete cube with weaker bounds. Given a weighting $w \in \mathbb{R}_{\geq 0}^n$ and $x,y \in \{0,1\}^n$ define the distance $d_w(x,y)$ to equal $\sum_{i=1}^{n} w_i \mathds{1}(x_i \neq y_i)$.

\begin{corollary}[Talagrand's Inequality with weaker bounds]
For every $\epsilon \geq 0$ and $\mathcal{A} \subset \{0,1\}^n$, at least a $1-\epsilon$ proportion of $x$ in $\{0,1\}^n$ have $\sup_{\norm{w}_2 = 1}d_w(x,\mathcal{A}) \leq \frac{2}{\mu(\mathcal{A})} \sqrt{\frac{ \log(\frac{2}{\epsilon})}{\mu(\mathcal{A})\epsilon}}$.
\end{corollary}

\begin{proof}
First we apply Theorem 1.1 with $\delta = \mu(\mathcal{A})/2$ giving us that $1-\epsilon$ proportion of $x$ have a vector of probabilities $P$ such that $\mathbb{P}(x+ Y_P \in \mathcal{A}) \geq \delta/2$ and $\norm{P}_2^2 \leq \frac{\log(\frac{2}{\epsilon})}{\mu(\mathcal{A})\epsilon}$. For every weighting $w$ of norm $1$ we have that $d_w(x+Y,x) = \sum_{i=1}^{n}Y_iw_i \leq \norm{Y}_2\norm{w}_2 = \norm{Y}_2$.
\\
\\
By Markov's inequality, for any $x$ and $P$ as above we have that $\mathbb{P}\big(\sum_{i=1}^{i=n} (Y_P)_iw_i >\frac{2}{\mu(\mathcal{A})} \sqrt{\frac{ \log(\frac{2}{\epsilon})}{\mu(\mathcal{A})\epsilon}}\big) < \mu(\mathcal{A})/2$. So there exists a $Y$ such that  $x + Y$ is in $\mathcal{A}$ and $\sum_{i=1}^{i=n} Y_iw_i \leq \frac{2}{\mu(\mathcal{A})} \sqrt{\frac{ \log(\frac{2}{\epsilon})}{\mu(\mathcal{A})\epsilon}}$. Hence for all weightings $w$, $d_w(x,\mathcal{A}) \leq \frac{2}{\mu(\mathcal{A})} \sqrt{\frac{ \log(\frac{2}{\epsilon})}{\mu(\mathcal{A})\epsilon}}$.
\end{proof}

\section{Large Intersection Problems}
In this section we use Theorem 1.1 to obtain non-trivial upper bounds for two intersection problems. The proof strategy for both of them will be the same. If $\mathcal{A}$ is a collection of sets that intersect in elements of an up-set $\mathcal{C}$ then we have that $\mathcal{A}^c+\mathcal{C}^c$ cannot intersect $\mathcal{A}$, for if it did then we would have $y \in \mathcal{A}$ and $z \notin \mathcal{C}$ such that $y^c + z \in \mathcal{A}$. But $y \cap (y^c+z) = y \cap z \subset z$ and since $\mathcal{C}$ is an up-set we have that $y \cap (y^c +z) \notin \mathcal{C}$. Let $\gamma$ and $C$ be as in Corollary 1.2 to ensure that whenever $\mu(\mathcal{B}) \geq \epsilon$ and $S$ is $(\gamma,C)$-well-spread we have that $\mu(\mathcal{B}+S) > 1-\epsilon$. Hence, if $\mathcal{C}^c$ is $(\gamma,C)$-well-spread we may conclude that $\mu(\mathcal{A}) \leq \epsilon$ for every $\mathcal{A}$ that is $\mathcal{C}$-intersecting.
\\
\\
In both of the following theorems we will not optimise for constants in two respects. We do not aim to minimise the length of the arithmetic progression for which we get a non-trivial upper bound or the size of the bipartite graph, as that is the purpose of the next section. Nor do we get the best possible constants as the length of the arithmetic progression (and size of the graph) grows. In the last section we will see that our methods have a natural barrier and are unable to prove anything faster than polynomial decay (in terms of the size of the AP and graph). Since the Erdős-Ko-Rado construction gives exponential decay and we believe that this is much closer to the truth, we do not think there is interest in optimising the polynomial decay.

\begin{theorem}
For $k$ sufficiently large, $\mu(\mathcal{A}) \leq k^{-1/3}$ for every $k$-AP-intersecting family $\mathcal{A}$.
\end{theorem}

\begin{proof}
We need only prove that the $k$-AP-free subsets are $(\gamma,C)$-well-spread for $\gamma = k^{-1/3}$ and $C = \frac{2}{3}k^{2/3}\log{k}$ as per the discussion at the start of the section. 
\\
\\
For a specific vector of probabilities $P$ with $\left \| P \right \|_2 ^2 \leq C$, let $I_m = \{i | P_i > 1/2^m\}$. We can simply bound the expected number of $k$-APs. Let $K$ denote the number of $k$-APs in $Y_P$. Note that the upper bound on the sum of the squares implies that $I_m$ has size at most $2^{2m}C$. Define $K_m$ to be the set of $k$-APs in $[n]$ such that they have at least $k/2$ elements in $I_m$ and less than $k/2$ elements in $I_{m-1}$. Note that all progressions with a non-zero probability of being a subset of $Y_P$ appear in some $K_m$. By definition, the contribution of and element of $K_m$ to the expected value of $K$ is at most $(1/2^{m-1})^{k/2}$. Also note that $K_m$ is empty for $m$ such that $2^{2m} < k/2C$. So we can bound $\mathbb{E}(K)$ as follows:
$$
\mathbb{E}(K) \leq \sum_{m \geq 1} (1/2^{m-1})^{k/2}|K_m|1_{\{2^{2m} \geq k/2S\}}
$$
To bound $|K_m|$, since the knowledge of the relative placing of two elements in an arithmetic progression fixes the others, we have that each pair of elements in $I_m$ is in at most $k^2$ arithmetic progressions together. So we have, rather crudely, that $|K_m| \leq k^2|I_m|^2 \leq 2^{4m}S^2$. Putting this into the expression above, noting that as $m$ increases by $1$ we at least halve the term $(1/2^{m-1})^{k/2}2^{4m}C^2$ for $k \geq 9$, we get that $\mathbb{E}(K)$ is at most twice the first non-zero term which is at most, after simplifying:
$$
16(4C/k)^{k/4}k^2
$$
Since $C < k/8$ for $k$ sufficiently large we may conclude.
\end{proof}

As our second application, we obtain non-trivial upper bounds for $m(K_{3,t})$ for $t$ sufficiently large.

\begin{theorem}
For $t$ sufficiently large, $m(K_{3,t}) \leq t^{-1/12}$.
\end{theorem}

\begin{proof}
Our starting point is identical to the previous proof. Now we have $\gamma = t^{-1/12}$ and $C = \frac{1}{6}t^{1/6}\log{t}$. We take a vector of probabilities $P$ with $\left \| P \right \|_2 ^2 \leq C$ and we stratify the edges into $I_m$ as before. Now instead of bounding the expected number of $K_{3,t}$, we will rather bound the expected number of `unpopular' $K_{3,s}$ where $s = \lfloor\frac{t}{2}\rfloor$. Here, we define a $K_{3,s}$ to be unpopular if every edge that it contains has $P_e < \sqrt{2C/t}$. Note that if $Y_P$ contains a $K_{3,t}$ then it contains an unpopular $K_{3,s}$. So we may look at edges in $I_m$ only for $m \geq \log_4{(t/2C})$.
\\
\\
Label a $K_{3,s}$ with $(m_1,m_2,m_3,m_1',...,m_s')$, to indicate for each vertex $v_i$ or $v_i'$ in the other half of the bipartition, the minimal $m$ for which they are incident with an edge from $I_m$. Now we upper bound how many $K_{3,s}$ can have a given label. We will consider the number of injections of $K_{3,s}$ into $K_n$ with each vertex, $v_i$/$v_i'$, in the $K_{3,s}$ being prescribed an edge adjacent to it and this edge being forced to go to an element of $I_{m_i}$/$I_{m_i'}$. We have at most $s^33^s$ choices for which edges to prescribe. We then have at most $\prod_{i=1,j=1}^{i=3,j=s} |I_{m_i}||I_{m_j'}|$ choices where to send the specialised edges. After this our hand is forced as we know where all the vertices go. For a given labelled copy $(m_1,m_2,m_3,m_1',...,m_s')$ we have that the probability it is in $Y_P$ is at most

$$
\prod_{i=1,j=1}^{i=3,j=s}2^{-t(m_i-1)/6-5(m_i'-1)/2},
$$

using the fact that a vertex labelled with $m_i$ has all probabilities incident with it at most $2^{-m_i+1}$ and splitting an edges' probability into $P_e^{5/6}$ and $P_e^{1/6}$ when incident to one of the one of the $v_i'$ and one of the $v_i$ respectively. So we obtain an upper bound for the expected number of these $K_{3,s}$, summing over all possible choices of labelling, of

$$
\sum_{(m_1,...,m_s')}s^33^s \prod_{i=1,j=1}^{i=3,j=s}2^{-t(m_i-1)/6-5(m_i'-1)/2} |I_{m_i}| |I_{m_i'}|
$$

For $t$ sufficiently large this is at most: 

\begin{align*}
s^33^s (\sum_{m\geq \log_4{(t/2C)}} |I_m|2^{-5(m-1)/2})^{s+3} \leq s^33^s (\sum_{m\geq \log_4{(t/2C)}} C2^{2m}2^{-5(m-1)/2})^{s+3},
\end{align*}

using that only edges in $I_m$ appear for $m \geq \log_4{(t/2C)}$. Note that $\sum_{m\geq \log_4{(t/2S)}} C2^{2m}2^{-5(m-1)/2}$ is at most $50C(2C/t)^{1/4}$, for $t$ sufficiently large. So we can conclude, for all $t$ large enough that $50C(2C/t)^{1/4} \leq 1/4$. So we obtain the desired bound as in the previous theorem.
\end{proof}

\section{The $C_4$-intersection problem}

In this section we prove non-trivial upper bounds for the measure of a cycle-intersecting family. We also prove an upper bound for $4$-AP-intersecting families. To do this we prove a version of Theorem 1.1 in the regime where $C$ is much smaller than 1. The way we use this to obtain upper bounds for the intersection problems is identical to the previous section.
\\
\\
In the proof we will use the following concentration inequality of Freedman \cite{Freedman}.
\begin{theorem*}
Consider a real-valued martingale $\{M_k : k = 0,1,2,\dots\}$.
\[
M_k-M_{k-1} \leq R \quad \text{almost surely for } k = 1,2,3,\dots
\]

Define the following:
\[
A_k := \sum_{j=1}^{k} \mathbb{E}_{j-1}[(M_j-M_{j-1})^2], \quad k = 1,2,3,\dots
\]

Then, for all $t \ge 0$ and $\sigma^2 > 0$,
\[
\mathbb{P}\Big( \exists k \ge 0 : M_k \ge t \text{ and } A_k \le \sigma^2 \Big)
\le \exp\Bigg( - \frac{t^2}{2(\sigma^2 + Rt/3)} \Bigg)
\]
\end{theorem*}

In this section we do not specify constants as they are small enough that there is no interest in what exactly we have obtained. We have a slightly strange statement and choice of constants in Theorem 3.1 as it is tailored to prove a non-trivial upper bound for the intersection problems.

\begin{theorem}
There exists an $\epsilon > 0$ and a $K > 0$ such that following holds: for every family $\mathcal{A} \subset \{0,1\}^n$ with $\mu(\mathcal{A}) \geq \frac{1}{2} - \epsilon$, for at least a $\frac{1}{2} + \epsilon$ proportion of elements $x \in \{0,1\}^n$ we have a $P \in [0,1]^n$ such that $\mathbb{P}(x+Y_P \in \mathcal{A}) > (2K)^\frac{3}{2}/(1-(2K)^\frac{1}{2})$ and $\norm{P}_2^2 \leq K$.
\end{theorem}

\begin{proof}
We will dispense with the suffix $X$ notation since we will not use the notation needed for probability with respect to $Y_P$ as $\mathbb{P}_P(x+Y_P) \in \mathcal{A})$ is just $W_n$. Consider the same submartingale considered in Theorem 1.1 with a given $C$, taken to be sufficiently small. Let $K = 100C$ and let $\epsilon$ be sufficiently small in terms of $C$. It will now be useful to consider the martingale $M_i = W_i - CO_i^2$ (we take $C < 1$ so $P_{i} < O_i \leq 1$ for all $i$). We will also let $D = \mathbb{E}(\sum_{i=1}^{n}{2O_iP_i}) = \mathbb{E}(\sum_{i=1}^{n}{CO_i^2})$, i.e. the expected drift upwards. Note that $|M_{i+1}-M_i|$ is always at most $\frac{1}{2}$, and that it is upper bounded by $O_i$ so we have that the expected value of the sum of squared differences, $\mathbb{E}(\sum_{i=0}^{n} (M_{i+1}-M_i)^2)$, is at most $D/C$. Also recall that $\mathbb{E}(\sum_{i=1}^{n} P_i^2) = DC/2$. We prove our result by considering two cases depending on the size of $D$. If $D$ is small then $\mathbb{E}(\sum_{i=0}^{n} (M_{i+1}-M_i)^2)$ is small so we can show that $M_n$ cannot deviate from $M_0$ by too much. If $D$ is large then we have that $\mathbb{E}(W_n)$ is large enough that we can crudely bound $\mathbb{P}(W_n \leq (2K)^\frac{3}{2}/(1-(2K)^\frac{1}{2}))$.
\vspace{1em}

\noindent\underline{Case I: $D < 0.01C$}
\vspace{0.6em}
\\
Here we have  $\mathbb{E}(\sum_{i=0}^{n} (M_{i+1}-M_i)^2) < 0.01$ and $|M_{i+1}-M_{i}|$ is bounded by $\frac{1}{2}$ so we apply our result of Freedman with the martingale $-M_i + M_0$ to obtain that for $\epsilon$ sufficiently small:

$$\mathbb{P}(M_n \leq 1/10 \text{ } and \text{ } \sum_{i=0}^{n} (M_{i+1}-M_i)^2 \leq 0.1) \leq e^{-(4/5 -\epsilon)^2/(0.2+4/15)} \leq 1/3$$ 

Then we have, by Markov's inequality applied to $\sum_{i=0}^{n} (M_{i+1}-M_i)^2$, that  $\mathbb{P}(M_n \leq 1/10) \leq 13/30$. Having taken $C$ sufficiently small and using that $W_n \geq M_n$ as $\mathbb{E}(\sum_{i=0}^{n}P_i^2) = DC/2$ we can simply apply Markov's inequality and a union bound as in Theorem 1.1 to conclude.
\vspace{1em}

\noindent\underline{Case II: $D \geq 0.01C$}
\vspace{0.6em}
\\
Let $(2K)^\frac{3}{2}/(1-(2K)^\frac{1}{2}) = K'$. Here we have that $\mathbb{E}(W_n) = \frac{1}{2}-\epsilon + D$ so $\mathbb{P}(W_n \leq K')K'+ 1-\mathbb{P}(W_n \leq K') \geq \frac{1}{2}-\epsilon + D$. That implies that $\mathbb{P}(W_n \leq K') \leq \frac{\frac{1}{2}+\epsilon-D}{1-K'}$. We have that $\mathbb{P}(\sum_{i=1}^{n} P_i^2 > K) < CD/2K$ by Markov's inequality. So we have that the probability that $\{\sum_{i=1}^{n} P_i^2 > K\}$ or $\{W_n \leq K'\}$ is less than:
$$
CD/2K + \frac{\frac{1}{2}+\epsilon-D}{1-K'}
$$
With our choice of $K$ this equals:
$$
\frac{D}{200} + \frac{\frac{1}{2}+\epsilon-D}{1-K'}
$$
\\
For $C$ sufficiently small we have that $\frac{1}{1-K'} \leq 1 + D/100$. So the above expression is at most:
$$
\frac{1}{2}+\epsilon - \frac{D}{4}
$$
Using that $D \geq 0.01C$, for $\epsilon$ sufficiently small in terms of $C$, this is at most $\frac{1}{2}-\epsilon$  so we may conclude.
\end{proof}

We now will easily deduce that every cycle-intersecting family of graphs is of measure at most $\frac{1}{2} - \epsilon$.

\begin{theorem}
There exists an $\epsilon > 0$ such that every cycle-intersecting family of graphs $\mathcal{A}$ has $\mu(\mathcal{A}) \leq \frac{1}{2} - \epsilon$.
\end{theorem}

\begin{proof}
Let $\epsilon$ and $K$ be as in Theorem 1.1. Then we have that if $\mu(\mathcal{A}) > \frac{1}{2} -\epsilon$, then there exists an $x \in \mathcal{A}$, and a $P \in [0,1]^{n(n-1)/2}$ with $\mathbb{P}(x^c+Y_P \in \mathcal{A}) > \frac{(2K)^{3/2}}{1-(2K)^{1/2}}$ and $\norm{P}_2^2 \leq K$. So if we show that $Y_P$ contains a cycle with probability at most $\frac{(2K)^{3/2}}{1-(2K)^{1/2}}$ we have that there exists a cycle-free graph $w$ and $x \in \mathcal{A}$ such that $x^c \cup w \in \mathcal{A}$, a contradiction.
\\
\\
So we bound the expected number of cycles in $Y_P$. For each vertex $v$, let $\alpha_v$ be the sum of $P_e^2$ over all edges incident to $v$. So $\sum_{v} \alpha_v \leq 2K$. Then the expected number of $t$-cycles in $Y_P$ is at most:
$$
\sum_{v_1,..,v_t}\prod_{i=1}^{t} P_{v_iv_{i+1}},
$$
taking indices modulo $k$. We bound this by repeatedly applying Cauchy-Schwarz on the sum, for all odd indices excluding $t$ if it is odd. We obtain, in the case that $t$ is even:
$$
\sum_{v_2,v_4,...,v_{2\lfloor\frac{t}{2}\rfloor}} \prod_{i=1}^{i=\lfloor\frac{t}{2}\rfloor} \alpha_{v_{2i}}
$$
This is then at most $(2K)^{t/2}$.
\\
\\
In the case that $t$ is odd we obtain:
$$
\sum_{v_2,..,v_{t-1}, v_t}\alpha_{v_{t-1}}^{1/2}\alpha_{v_t}^{1/2}P_{v_{t-1}v_t}\prod_{i=1}^{i=\lfloor\frac{t}{2}\rfloor-1} \alpha_{v_{2i}}
$$
Then applying Cauchy-Schwarz to the sum over $v_t$ we obtain at most:
$$
\sum_{v_2,..,v_{t-1}}\alpha_{v_{t-1}}(2K)^{1/2}\prod_{i=1}^{i=\lfloor\frac{t}{2}\rfloor-1} \alpha_{v_{2i}},
$$
which in turn is at most $(2K)^{t/2}$. So, summing over $t \geq 3$, we obtain at most $\frac{(2K)^{3/2}}{1-(2K)^{1/2}}$.
\end{proof}

Lastly, we prove a non-trivial upper bound for the $4$-AP-intersection problem.

\begin{theorem}
There exists an $\epsilon > 0$ such that every $4$-AP-intersecting family $\mathcal{A}$ has $\mu(\mathcal{A}) \leq \frac{1}{2} - \epsilon$.
\end{theorem}

\begin{proof}
Our proof is very similar to the previous proof so we only need to show that the expected number of $4$-APs in $Y_P$ is at most $\frac{(2K)^{3/2}}{1-(2K)^{1/2}}$. The expected number is equal to:
$$
\sum_{a,d} P_aP_{a+d}P_{a+2d}P_{a+3d},
$$
which by Cauchy-Schwarz is at most:
$$
\sum_{a} (\sum_{d}  P_a^2P_{a+d}^2)^{1/2}(\sum_{d}  P_{a+2d}^2P_{a+3d}^2)^{1/2},
$$
which in turn is at most:
$$
K^{1/2} \sum_{a} P_a (\sum_{d} P_{a+2d}^2P_{a+3d}^2)^{1/2},
$$
since $\sum_{d}P_{a+d}^2 \leq K$. Again using Cauchy-Schwarz we obtain:
$$
K(\sum_{a,d}P_{a+2d}^2P_{a+3d}^2)^{1/2},
$$
which is simply at most $K^2$. Then note that $K \leq 1$, so $K^2 \leq \frac{(2K)^{3/2}}{1-(2K)^{1/2}}$ and we may conclude.
\end{proof}

\section{Concluding Remarks and Further Research}
Our results have provided non-trivial upper bounds for the two classes of intersection problems we have considered. Returning to the problem posed by Alon to determine which graphs $H$ have $m(H) = \frac{1}{2}$, the graphs that our methods do not resolve are trees that are not a disjoint collection of stars. Similarly our methods do not prove any non-trivial upper bound for the $3$-AP intersection problem. There is good reason for this. Recall that for any $P= (C/\sqrt{n},...,C/\sqrt{n})$ we have that $Y_P$ contains no $3$-APs with probability tending to $0$, whereas in all our proofs of non-trivial upper bounds it is crucial that this probability is bounded below by a positive constant independent of $n$. The same applies to the $H$-intersection problem when $H$ is a tree. We believe these problems to be the most natural direction for further research. One could also consider the related problem of proving a version of Corollary 1.3 when we only have a $f(\norm{P}_{\infty})$ probability of $Y_P$ being in $S$ where $\lim_{x \to 0} f(x) = 0$.
\\
\\
Another direction we think is interesting would be trying to better understand the asymptotic behaviour of $m(K_{t,t})$ (and the analogous problem for $k$-APs). Our result shows it tends to zero polynomially fast, and again this is a natural barrier for any approach using Theorem 1.1. Take the arithmetic progression problem for example: we must have $\norm{P}_2^2 \leq k$ or else we could have $Y_P$ certainly containing a $k$-AP. The dependency between $\epsilon$, $\delta$ and the upper bound for $\norm{P}_2^2$ is such that this only implies a polynomial upper bound. So we ask the following natural pair of questions:

\begin{question}
Does there exist $c < 1$ such that the following holds: if $\mathcal{A}$ is a $k$-AP intersecting family then $\mu(A) < c^k$?
\end{question}

\begin{question}
Does there exist $c<1$ such that $m(K_{t,t}) < c^{t^2}$ for all $t$?
\end{question}

We do not know of an example that would rule out proving these results by improving the dependency between $\epsilon$, $\delta$ and the upper bound for $\norm{P}_2^2$ in Theorem 1.1.

\section*{Acknowledgements}
The author would like to thank Timothy Gowers for useful comments on an earlier draft of this paper.

\bibliographystyle{amsplain}
\bibliography{IntersectionProblems}
\end{document}